\documentclass[a4paper,12pt,oneside]{article}

\usepackage[utf8]{inputenc}
\usepackage[english]{babel}
\usepackage{amsthm}
\usepackage{amssymb}
\usepackage{amsmath}
\usepackage{hyperref}
\usepackage{mathrsfs}

\usepackage{pstricks}

 \usepackage{tikz}

\usepackage{graphicx}
\graphicspath{{images/}}

\usepackage{longtable,geometry}
\usepackage{color}

\usepackage{hyperref}

\newtheoremstyle{note}{12pt}{12pt}{}{}{\bfseries}{.}{.5em}{}
\title{\LARGE\textbf{A Phase Transition for Circle Maps with a Flat Spot and Different Critical Exponents}}
\author{ Liviana Palmisano and Bertuel Tangue}
\makeatletter
\newtheorem{theo}[equation]{Theorem}
\newtheorem{prop}[equation]{Proposition}

\newtheorem{defin}[equation]{Definition}

\newtheorem{rem}[equation]{Remark}
\newtheorem{cor}[equation]{Corollary}

\numberwithin{equation}{section}
\newtheorem{lem}[equation]{Lemma}

\usepackage{babel}

\newcommand{\N}{{\mathbb N}}
\newcommand{\Z}{{\mathbb Z}}
\newcommand{\R}{{\mathbb R}}

\renewcommand{\S}{{\mathbb S}^1}

\newcommand{\Cd}{{{\mathcal C}^2}}

\newcommand{\Cuno}{{{\mathcal C}^1}}

\renewcommand{\L}{{\mathscr L}}
\newcommand{\W}{{\mathscr W}}

\begin{document}
\maketitle
\author
\textcolor{blue}{}\global\long\def\sbr#1{\left[#1\right] }
\textcolor{blue}{}\global\long\def\cbr#1{\left\{  #1\right\}  }
\textcolor{blue}{}\global\long\def\rbr#1{\left(#1\right)}
\textcolor{blue}{}\global\long\def\ev#1{\mathbb{E}{#1}}
\textcolor{blue}{}\global\long\def\R{\mathbb{R}}
\textcolor{blue}{}\global\long\def\E{\mathbb{E}}
\textcolor{blue}{}\global\long\def\norm#1#2#3{\Vert#1\Vert_{#2}^{#3}}
\textcolor{blue}{}\global\long\def\pr#1{\mathbb{P}\rbr{#1}}
\textcolor{blue}{}\global\long\def\qq{\mathbb{Q}}
\textcolor{blue}{}\global\long\def\aa{\mathbb{A}}
\textcolor{blue}{}\global\long\def\ind#1{1_{#1}}
\textcolor{blue}{}\global\long\def\pp{\mathbb{P}}
\textcolor{blue}{}\global\long\def\cleq{\lesssim}
\textcolor{blue}{}\global\long\def\ceq{\eqsim}
\textcolor{blue}{}\global\long\def\Var#1{\text{Var}(#1)}
\textcolor{blue}{}\global\long\def\TDD#1{{\color{red}To\, Do(#1)}}
\textcolor{blue}{}\global\long\def\dd#1{\textnormal{d}#1}
\textcolor{blue}{}\global\long\def\eqdef{:=}
\textcolor{blue}{}\global\long\def\ddp#1#2{\left\langle #1,#2\right\rangle }
\textcolor{blue}{}\global\long\def\En{\mathcal{E}_{n}}
\textcolor{blue}{}\global\long\def\Z{\mathbb{Z}}
\textcolor{blue}{{} }

\textcolor{blue}{}\global\long\def\nC#1{\newconstant{#1}}
\textcolor{blue}{}\global\long\def\C#1{\useconstant{#1}}
\textcolor{blue}{}\global\long\def\nC#1{\newconstant{#1}\text{nC}_{#1}}
\textcolor{blue}{}\global\long\def\C#1{C_{#1}}
\textcolor{blue}{}\global\long\def\meas{\mathcal{M}}
\textcolor{blue}{}\global\long\def\cSpace{\mathcal{C}}
\textcolor{blue}{}\global\long\def\pspace{\mathcal{P}}

\begin{abstract} 
We study circle maps with a flat interval where the critical exponents at the two boundary points of the flat spot might be different. 
The space of such systems is partitioned in two connected parts whose common boundary only depends on the critical exponents. At this boundary there is a phase transition in the geometry of the system. 
Differently from the previous approaches, this is achieved by studying the asymptotical behavior of the renormalization operator. 
\end{abstract}

\section{Introduction}
The dynamics of circle maps with a flat interval has been intensively explored in the past years, see \cite{5aut, MartensPalmisano, P2, P3, TangermanVeermanI, TangermanVeermanII, Veerman}.
Among many reasons to study these maps, they appear as first return map of a special flow on the torus, called Cherry flow, see \cite{MartensdeMeloMendesvanStrien, Mendes, MoreiraGaspar, P1, P4, P5}.
Because of their connection with Cherry flows, the maps considered in the above cited papers have, near both boundary points of the flat interval, the form of $x^{\ell}$, where $\ell$ is a positive real number and it is called the critical exponent of the map. Moreover, in \cite{Graczyk}, the dynamics of maps with different critical exponents at the two boundary points of the flat interval is studied. More specifically, the author considers the special case of maps having one critical exponent equal to one and the other strictly larger than one. The dynamics of these systems has been essential in the study of bimodal circle maps, see for example \cite{CrovisierGuarinoPalmisano}.
 
We consider here the general case of maps with a flat interval and critical exponents  not necessarily equal.  In our context, the exponents can have all real values starting in one. The space of our maps, denoted by $\W$ contains circle maps with a flat piece, different critical exponents $\ell_1,\ell_2\ge 1$ and Fibonacci rotation number. See Subsection \ref{renormalization} for a precise definition. 
 
 As in \cite{Graczyk}, \cite{5aut} and \cite{ P1} we are interested in the study of the geometry of the non-wandering set of the system. This is the set obtained by removing from the circle all preimages of the flat interval and it turns out to be a Cantor set, see (\ref{Def:nonwandering}). 
 The small scale geometry of the map near to the boundary points of the flat interval gives global information on the geometry of the non-wandering set. Differently from the previous works, information on the geometry of the system near the flat interval is obtained by the study of the asymptotical behavior of the renormalization operator. One can think of the renormalization operator as a microscope. We cut off a neighborhood around the image of the flat piece and we look at the first return map to this neighborhood. After rescaling and flipping, the first return map is again a map of our class. This process, that informally describes the action of the renormalization operator, asymptotically, gives information on the small scale geometry of our systems near the flat interval. The action of the renormalization operator on our class of maps is explained in Subsection \ref{renormalization}. 
 
We denote by $w_n$ the quadruple of the relevant scaling ratios describing the $n^{th}$ renormalization of a map. Its asymptotical behavior is described by a $4\times 4$ matrix having eigenvalues $\lambda_u>0$, $\lambda_s\in(0,1)$, $\lambda_1=1$ and $\lambda_0=0$. The matrix, and its eigenvalues, only depend on the critical exponents. In particular, expressing the vector $w_n$ of scaling ratios of consecutive renormalizations in the basis of the four eigenvectors $E_u$, $E_s$, $E_1$, $E_0$ one has, for all $n\in\N$, the coordinates $C_u(n)$, $C_s(n)$, $C_1(n)$, $C_0(n)$. The coordinate $C_u(n)$, related to the eigenvalue $\lambda_u$, effects the asymptotical behavior of the renormalization operator and as consequence, the geometry of the system. We say that the sequence of renormalizations is bounded if the  $\limsup_{n\to\infty}C_u(n)$ is finite and it is degenerate if the $\limsup_{n\to\infty}C_u(n)$ is infinite, see Definition \ref{def:boundeddegenerate}. 

Furthermore, there exists a curve $\Gamma$, defined by the equation $\lambda_u(\ell_1, \ell_2)=1$, which separates the $\ell_1,\ell_2$-plane in two connected components, $Q_-$ and $Q_+$. We detect a change of the geometry of the system while crossing $\Gamma$. This is presented in our main theorem, see Theorem \ref{Theo:maintheo} and summarized in the following. 
 
\paragraph{Theorem A.}
Let $f\in\W$ then the following holds.
\begin{itemize}
\item[-] If $\ell_1,\ell_2\in Q_+$ then the sequence of renormalizations is bounded.
\item[-] If $\ell_1,\ell_2\in Q_-$ then either the sequence  of renormalizations is bounded or it is degenerate. In particular, if $C_u(0)$ is large enough then the sequence of renormalizations is degenerate.
\end{itemize}

\bigskip

In the space of maps whose renormalizations have first coefficient of the unstable eigenvector large enough, namely $C_u(0)>>1$, we detect a phase transition in the geometry of the system from degenerate to bounded. The transition occurs along a curve depending only on the two critical exponents. We would like to stress that maps with this property are very easy to realize: it suffices to start with a large enough flat interval.
Moreover, in one of the two connected components, we also detect a dichotomy in the geometry which can be compared with the one found for Lorenz maps in \cite{MartensWinckler}.

Observe that our result contains also the case of maps with the same critical exponents for which the same questions has been studied in \cite{5aut}. In that context, a transition in the geometry of the system is found when the critical exponent crosses $2$ with no further assumptions. Our class contains also maps with one critical exponent equal to one and the other one larger than one, studied in \cite{Graczyk} and for which the degenerate geometry always holds. Supported by the above cited cases, we expect that the dichotomy  in $Q_-$ reduces to degeneracy only.

The consequences of the asymptotical behavior of the renormalizations on the geometry of the system are explained in Theorem \ref{Prop:consequencesgeometry}. When the renormalizations are bounded, the Hausdorff dimension of the non wandering set is strictly positive. While, when the renormalizations are degenerate, the Hausdorff dimension of the non-wandering set is zero. 
Moreover when the renormalizations are degenerate one can define the following quantity
$$
G_u(f):=\lim\frac{C_u(n)}{\lambda_u^n}\in\R
$$
and obtain a very explicit expression describing the divergence of the renormalizations. This is summarized in the following theorem. For more details refer to Theorem \ref{Prop:consequencesgeometry}. 

\paragraph{Theorem B.} Let $f\in\W$.
If the sequence of renormalizations is bounded, then 
\begin{itemize}
\item[-]$|w_n|$ is bounded,
\item[-] the non-wandering set has strictly positive Hausdorff dimension.
\end{itemize}
If the sequence of renormalizations is degenerate, then $\lambda_u\ge 1$. In particular, if $\lambda_u>1$ then  
\begin{itemize}
\item[-]$w_{2n}=G_u(f)\lambda_u^n\left(E_u+o(1)\right)$,
\item[-] $G_u(f)<0$,
\item[-] the non-wandering set has zero Hausdorff dimension.
\end{itemize}

\bigskip

As final remark we would like to stress that our discussion gives a method to study the geometry of the attractor of a system. This is achieved by the study of the asymptotics of the renormalization operator.  We believe that such a method can be applied in very general and different contexts, such as circle maps with discontinuities, Fibonacci unimodal maps, Lorenz maps, etc. where one can allow these maps to have different critical exponents. Quadratic unimodal Fibonacci maps have been successfully studied in \cite{LyubichMilnor}. For period doubling unimodal maps with different critical exponents recent results were obtained in 
 \cite{KozlovskivanStrien}.

\paragraph {Standing notation.} Let $\alpha_n$ and $\beta_n$ be two sequences of positive numbers. We say that  $\alpha_n$ is of the order of $\beta_n$ if there exists an uniform constant $K>0$ such that  $\alpha_n\le K \beta_n$. We will use the notation $\alpha_n=O(\beta_n).$
Moreover we denote by $[a, b] = [b, a]$ the shortest interval
between $a$ and $b$ regardless of the order of these two points. The length of that interval
in the natural metric will be denoted by $\left|[a , b]\right|$. 

\paragraph{Acknowledgements.}
The first author is supported by the Trygger Foundation. The second author is supported by the Centre d'Excellence Africain en Science Mathématiques et Applications (CEA-SMA).

\section{Class of Maps and Renormalization}
\label{section:renorm}
In this section we introduce the dynamical systems of interest, namely  circle maps with a flat interval which has possibly different critical exponents at the boundary points. Furthermore, we describe the action of the renormalization operator on such a class of maps.
\subsection{The class of maps}\label{ourfunctions}
We fix two real positive numbers $\ell_1,\ell_2\ge 1$ and we denote by $\Sigma^{(X)}$ the simplex
\begin{equation*}
\Sigma^{(X)}=\{(x_1,x_2,x_3,x_4,s)\in\mathbb R^5 | x_1<0<x_3<x_4<1, 0<x_2<1\text{ and } 0<s<1 \}
\end{equation*}
and by $\text{ Diff }^2([0,1])$ the space of $\Cd$, orientation preserving diffeomorphisms of $[0,1]$.
The space of circle maps with a flat interval and different exponents at the boundary points, is described by 
$$\L^{(X)}=\Sigma^{(X)}\times \text{Diff }^2([0,1])\times \text{Diff }^2([0,1])\times \text{Diff }^2([0,1]).$$ 
The space $\L^{(X)}$ is equipped by a distance which is the sum of the usual distances: the euclidian distance on $\Sigma^{(X)}$ 
and the sum of the $\Cd$ distance, $|\cdot |_{\Cd}$, on $\text{Diff }^2([0,1])$.

 A point $f=(x_1,x_2,x_3,x_4, s,\varphi,\varphi^{l},\varphi^{r})\in \L^{(X)}$ represents the following interval map $f:[x_1,1]\to[x_1,1]$ defined by

\begin{center}
\begin{equation}\label{eqfun}
f(x):=\left\{ 
    \begin{array}{ll}
   
      (1-x_2)q_s\circ\varphi\left(1-\frac{x}{x_1}\right)+x_2 & \text{ if } x\in[x_1,0) \\
      
      x_1\left(\varphi^{l}\left(\frac{x_3-x}{x_3}\right)\right)^{\ell_1} & \text{ if } x\in[0,x_3] \\
       
       0 & \text{ if } x\in(x_3,x_4)\\
       
       x_2\left(\varphi^{r}\left(\frac{x-x_4}{1-x_4}\right)\right)^{\ell_2} & \text{ if } x\in[x_4,1] \\ 
      
    \end{array}
\right.
\end{equation}
\end{center}
where $q_s:[0,1]\to [0,1]$ is a diffeomorphic part of $x^{\ell_2}$ parametrized by $s\in (0,1)$, namely $$q_s(x)=\frac{\left[(1-s)x+s\right]^{\ell_2}-s^{\ell_2}}{1-s^{\ell_2}}.$$
The real numbers $\ell_1,\ell_2\ge 1$ are called the critical exponents of $f$. The role of $q_s$ will become clear in the study of the asymptotical behavior of renormalization, see Section \ref{asymrenormalization}. In particular, when the consecutive renormalizations of a map diverge then the renormalizations will develop in the limit a critical point at $x=x_1$. In this case the coordinate $s$ will tend to zero. The reader can keep in mind Figure \ref{Fig1}.

\begin{figure}[h]
\label{Fig1}
\centering
\includegraphics[width=0.6\textwidth]{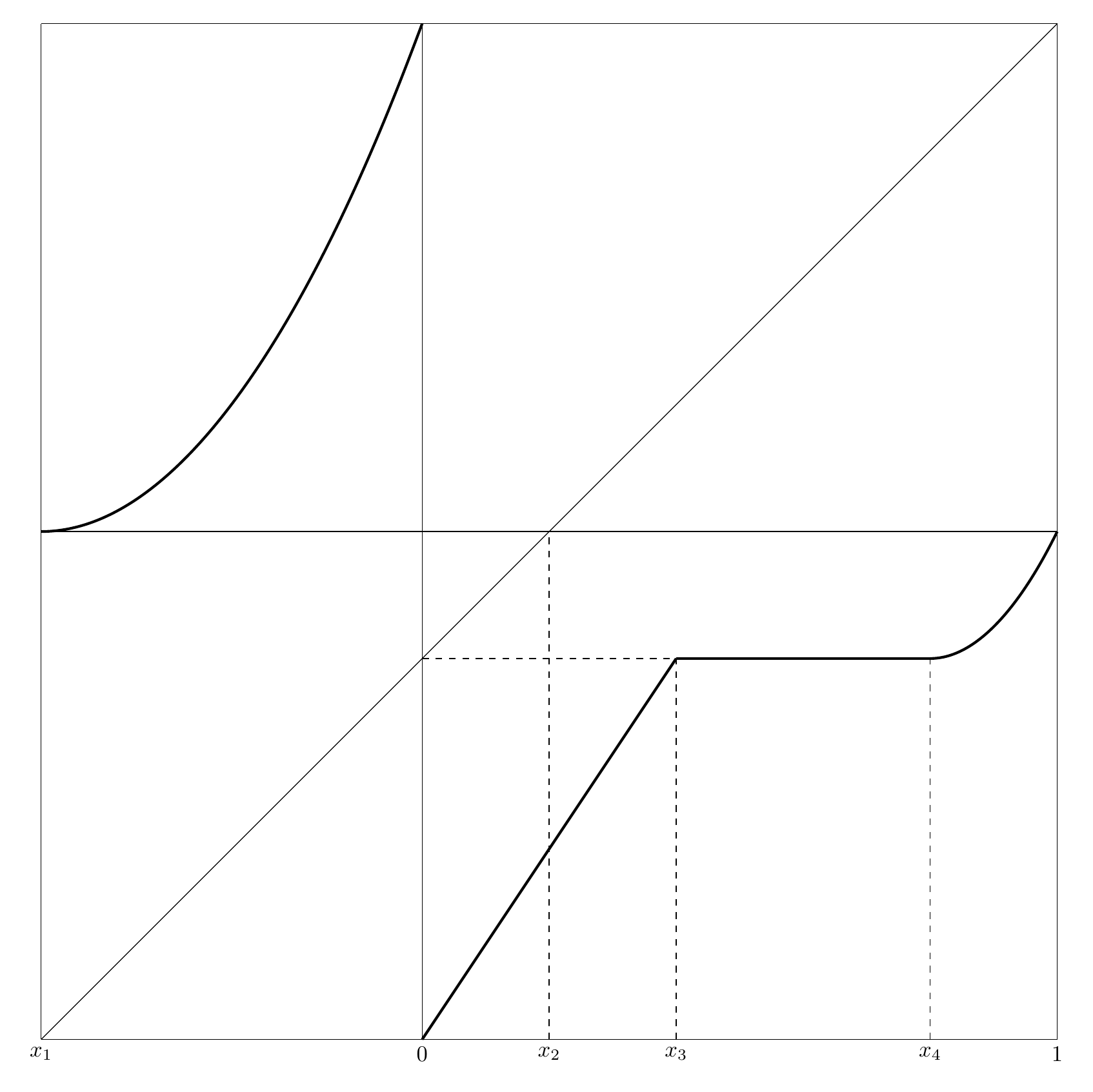}
\caption{A function in $\L^{(X)}$}
\end{figure}

Depending on the situation we will use different coordinate systems. Given a system $f=(x_1,x_2,x_3,x_4, s, \varphi,\varphi^{l},\varphi^{r})\in \L^{(X)}$, we will represent it in $S-$coordinates as follows: $f=(S_1,S_2,S_3,S_4, S_5, \varphi,\varphi^{l},\varphi^{r})$ where

  \begin{eqnarray*}
  S_1=\frac{\left|\left[x_2,x_3\right]\right|}{\left|\left[0,x_3\right]\right|}=\frac{x_3-x_2}{x_3},\text{  }  S_2=\frac{\left|\left[x_4,1\right]\right|}{\left|\left[x_2,1\right]\right|}=\frac{1-x_4}{1-x_2}, \text{  } S_3=\frac{\left|\left[0,x_3\right]\right|}{\left|\left[x_4,1\right]\right|}=\frac{x_3}{1-x_4},
  \end{eqnarray*}
   \begin{eqnarray*} 
  S_4=\frac{\left|\left[0, x_2\right]\right|}{\left|\left[x_1,0\right]\right|}=\frac{x_2}{-x_1},\text{  } S_5=s^{\ell_2-1}.
 \end{eqnarray*}
As a consequence we define 
\begin{equation*}
\Sigma^{(S)}=\{(S_1,S_2,S_3,S_4,S_5)\in\mathbb R^5 | 0<S_2,S_5<1 \text{ and } 0<S_3, S_4 \}
\end{equation*}
and 
$$\L^{(S)}=\Sigma^{(S)}\times \text{Diff }^2([0,1])\times \text{Diff }^2([0,1])\times \text{Diff }^2([0,1]).$$ 
Similarly, given a system $f=(S_1,S_2,S_3,S_4, S_5, \varphi,\varphi^{l},\varphi^{r})\in \L^{(S)}$ we will represent it in $Y-$coordinates as follows: $f=(y_1,y_2,y_3,y_4, y_5, \varphi,\varphi^{l},\varphi^{r})$ where
$$
  \begin{matrix}
 y_1=S_1, & y_2=\log S_2, & y_3=\log S_3, & y_4=\log S_4, & y_5=\log S_5
\end{matrix}
$$
Also in this case we define 
\begin{equation*}
\Sigma^{(Y)}=\{(y_1,y_2,y_3,y_4,y_5)\in\mathbb R^5 |  y_2, y_5<0 \}
\end{equation*}
and 
$$\L^{(Y)}=\Sigma^{(Y)}\times \text{Diff }^2([0,1])\times \text{Diff }^2([0,1])\times \text{Diff }^2([0,1]).$$ 

Observe that these changes of coordinates induce diffeomorphisms between $\L^{(X)}$, $\L^{(S)}$ and $\L^{(Y)}$. In particular, by explicit calculations the following lemma holds. 
\begin{lem}\label{xtos}
The inverse of $(x_1, x_2, x_3, x_4)\to (S_1, S_2, S_3, S_4)$ is given by
\begin{enumerate}
 \item $x_1=-\frac{S_3(1-S_1)S_2}{(1+S_3(1-S_1)S_2)S_4}$,
 \item $x_2=\frac{S_3(1-S_1)S_2}{1+S_3(1-S_1)S_2}$,
 \item $x_3=\frac{S_3S_2}{1+S_3(1-S_1)S_2}$,
  \item $x_4=1-\frac{S_2}{1+S_3(1-S_1)S_2}$.
\end{enumerate}
\end{lem}

From the context and the notation it will be clear which parametrization of our space we are using. The space will then be simply denoted by $\L$ instead of $\L^{(X)}$, $\L^{(S)}$ or $\L^{(Y)}$.

\subsection{Renormalization}\label{renormalization}
In this section we are going to define the renormalization operator. The renormalization scheme that we are going to use is adapted to study circle maps with Fibonacci rotation number, see Subsection \ref{Fib}. For basic concepts concerning circle maps see \cite{deMelovanStrien}.
\begin{defin}
A map $f\in\L$ is renormalizable if $0<x_2<x_3$. The space of renormalizable maps will be denoted by $\L_0$.
\end{defin}
Let $f\in\L_0$ and let $\text{pre}R(f)$ be the first return map of $f$ to the interval $[x_1,x_2]$. Let us consider the function $h:[x_1,x_2]\to[0,1]$ 
defined as $h(x)=\frac{x}{x_1}$ for all $x\in [x_1,x_2]$. Then the function 
\begin{equation*}
Rf:=h\circ \text{pre}R(f)\circ h^{-1}
\end{equation*}
is again a map of $\L$. Notice that $Rf$ is nothing else than the first return map of $f$ to the interval $[x_1,x_2]$ rescaled and flipped. This defines the renormalization operator 
$$R:\L_0\to \L.$$ 
\begin{defin}
A map $f\in\L$ is $\infty$-renormalizable if for every $n\geq 0$, $R^n f\in\L_0$. The set of $\infty$-renormalizable functions will be denoted by $\mathscr W\subset\L$. The maps in $\W$ are called Fibonacci maps. 
\end{defin}
\begin{rem}
Observe that, if $f\in\W$, by identifying $x_1$ with $1$ we obtain a map of the circle having Fibonacci rotation number.
\end{rem}
We also introduce, for $\delta>0$, the subset $\W_{\delta}$ of $\W$ consisting of  the Fibonacci maps with bounded diffeomorphisms, namely
$$
\W_{\delta}=\left\{f\in\W\left|\right. \left|\varphi\right|_{\Cd}+\left|\varphi^l\right|_{\Cd}+\left|\varphi^r\right|_{\Cd}<\delta\right\}.
$$
In Definition \ref{zoomop} we introduce the concept of the zoom operator needed later to describe the action of the renormalization operator on the space of diffeomorphisms.
\begin{defin}\label{zoomop}
Let $I=[a,b]\subset[0,1]$. The \emph{zoom} operator $Z_{I}:\text{ Diff }^0([0,1])\to \text {Diff }^0([0,1])$ is defined as
$$
Z_{I}\varphi(x)=\frac{\varphi((b-a)x+a)-\tilde a}{\tilde b-\tilde a}
$$
where $\varphi\in \text{ Diff }^0([0,1])$, $x\in[0,1]$, $\tilde a=\varphi(a)$ and $\tilde b=\varphi(b)$.
\end{defin}
 The following two lemmas are a direct consequence of the definition of the renormalization operator.
\begin{lem}\label{changexs}
Let $f=(x_1,x_2,x_3,x_4, s, \varphi,\varphi^{l},\varphi^{r})\in \L_0$ and let\\
$Rf=(\tilde x_1,\tilde x_2,\tilde x_3, \tilde x_4, \tilde s, \tilde \varphi,\tilde\varphi^{l},\tilde\varphi^{r})$. Then
\begin{eqnarray*}
 1.&\tilde x_{1}&=\frac{x_2}{x_1}\\
 2.&\tilde x_{2}&=\left(\varphi^{l}\left(\frac{x_3-x_2}{x_3}\right)\right)^{\ell_1}\\
 3.&\tilde x_3&=1-\left(\varphi^{-1}\circ q_s^{-1}\right)\left(\frac{x_4-x_2}{1-x_2}\right)\\
  4.&\tilde x_4&=1-\left(\varphi^{-1}\circ q_s^{-1}\right)\left(\frac{x_3-x_2}{1-x_2}\right)\\
  5.&\tilde s&=\varphi^{l}\left(\frac{x_3-x_2}{x_3}\right)\\
 6.&\tilde \varphi &=Z_{\left[\frac{x_3-x_2}{x_3},1\right]}\varphi^{l}\\
   7.& \tilde\varphi^{l}&=\varphi^{r}\circ Z_{\left[1-\tilde x_3,1\right]}\left( q_s\circ \varphi\right)\\
   8.&\tilde \varphi^{r}&=Z_{\left[0, \frac{x_3-x_{2}}{x_3}\right]}\varphi^{l}\circ Z_{\left[0,1-\tilde x_4\right]} \left(q_s\circ \varphi\right)
 \end{eqnarray*}
\end{lem}

\begin{lem}\label{ss}
Consider a map $f=(S_1,S_2,S_3,S_4, S_5, \varphi,\varphi^{l},\varphi^{r})\in \L_0$ and its renormalization
$Rf=(\tilde S_1, \tilde S_2, \tilde S_3, \tilde S_4, \tilde S_5, \tilde \varphi,\tilde\varphi^{l},\tilde\varphi^{r})$. Then
\begin{eqnarray*}
1.&\tilde S_{1}&=1-\left(\frac{\ell_2 S_{1}^{\ell_1}}{S_{2}}\right)\cdot\left[\frac{S_2}{1-\left(\varphi^{-1}\circ q_s^{-1}\right)\left(1- S_2\right)}\cdot\frac{\left(\varphi^{l}\left( S_1\right)\right)^{\ell_1}}{\ell_2 S_{1}^{\ell_1}}\right]\\
2.&  \tilde S_{2}&=\frac{ S_{1}S_{2}S_{3}}{\ell_2 S_{5}}\cdot\left[\frac{\ell_2 S_5 \left(\varphi^{-1}\circ q_s^{-1}\right)\left(S_{1}S_{2}S_{3}\right)}{S_{1}S_{2}S_{3}}\cdot\frac{1}{1-\left(\varphi^{l}\left( S_1\right)\right)^{\ell_1}}\right]\\
3.&  \tilde S_{3}&=\frac{S_{5}}{ S_{1}S_{3}}\cdot\left[\frac{S_{1}S_{3} \left(1-\left(\varphi^{-1}\circ q_s^{-1}\right)\left(1-S_{2}\right)\right)}{S_5 \left(\varphi^{-1}\circ q_s^{-1}\right)\left(S_{1}S_{2}S_{3}\right)}\right]\\
 4.& \tilde S_{4}&=\frac{S_{1}^{\ell_1}}{S_{4}}\cdot\left[\left(\frac{\varphi^{l}\left( S_1\right)}{S_{1}}\right)^{\ell_1}\right]\\
5.&  \tilde S_{5}&=S_{1}^{\ell_1-1}\cdot\left[\left(\frac{\varphi^{l}\left( S_1\right)}{S_{1}}\right)^{\ell_1 -1}\right]\\
6.&\tilde \varphi &=Z_{\left[S_1,1\right]}\varphi^{l}\\
7.& \tilde\varphi^{l}&=\varphi^{r}\circ Z_{\left[q_s^{-1}\left(1-S_{2}\right),1\right]}\left( q_s\right)\circ Z_{\left[\varphi^{-1}\circ q_s^{-1}\left(1-S_{2}\right),1\right]}\left( \varphi\right)\\
8.& \tilde\varphi^{r}&=Z_{\left[0, S_1\right]}\left(\varphi^{l}\right)\circ Z_{\left[0,q_s^{-1}\left(S_1S_{2}S_3\right)\right]} \left(q_s\right)\circ Z_{\left[0,\varphi^{-1}\circ q_s^{-1}\left(S_1S_{2}S_3\right)\right]} \left(\varphi\right)
\end{eqnarray*}
\end{lem}
\subsection{Fibonacci rotation number}\label{Fib}

Let $U_f=[x_3, x_4]$ be the flat interval of $f$. Observe that $R^nf_{|[0,1]}$ is a rescaled version of $f^{q_n}$ where the sequence $\left(q_n\right)_{n\in\N}$ is the Fibonacci sequence satisfying: $q_1=1$, $q_2=2$ and for all $n\geq 3$, $q_n=q_{n-1}+q_{n-2}$.

 Moreover if $R^nf=(x_{1,n},x_{2,n},x_{3,n},x_{4,n}, s_n,\varphi_{n},\varphi_{n}^l,\varphi_{n}^r)$ then the points $x_{1,n},x_{2,n},x_{3,n},x_{4,n}$ correspond to dynamical points of the original function $f$. Namely, 
\begin{itemize}
\item[-] $\hat x_{1,n}=f^{q_n+1}(x_3)=f^{q_n+1}(x_4)=f^{q_n}(0)$,
\item[-]  $\hat x_{2,n}=f^{q_{n+1}+1}(x_3)=f^{q_{n+1}+1}(x_4)=f^{q_{n+1}}(0)$,
\item[-]  $\hat x_{3,n}=f^{-q_{n}+1}(x_3)$,
\item[-]  $\hat x_{4,n}=f^{-q_{n}+1}(x_4)$.
\end{itemize}
\begin{prop}\label{Prop:gapsgotozero}
The sequence ${\left|\left[0,\hat x_{2,n}\right]\right|}$ tends to zero at least exponentially fast. 
\end{prop}
\begin{cor}\label{Cor:flatlonger}
Fix $\delta>0$. Then for all $f\in\W_{\delta}$ and $n\in\N$,
$$
\frac{\left|\left[U, f^{-1}\left(\hat x_{2,n}\right)\right]\right|}{\left|U\right|}=\frac{\left|\left[x_3, f^{-1}\left(\hat x_{2,n}\right)\right]\right|}{\left|[x_3,x_4]\right|}=O(1)
$$
where the constant in the order depends only on $\delta$.
\end{cor}

\begin{prop}\label{Prop:preimbiggap}
Fix $\delta>0$. Then for all $f\in\W_{\delta}$ and $n\in\N$,
$$
\max\left\{\left|\left[0, \hat x_{3,n}\right]\right|,\left|\left[\hat x_{4,n}, \hat x_{3,n-2} \right]\right|\right\}=O\left({\left|\left[\hat x_{3,n},\hat x_{4,n}\right]\right|}\right)
$$
and 
$$
\left|\left[0, \hat x_{2,n}\right]\right|=O\left({\left|\left[\hat x_{3,{n+1}},\hat x_{4,{n+1}}\right]\right|}\right)
$$
where the constant in the orders depend only on $\delta$.
\end{prop}
The proofs of Proposition \ref{Prop:gapsgotozero} and Proposition \ref{Prop:preimbiggap} are the same as the ones of Proposition $1$ and Proposition $2$ in \cite{5aut} where the authors prove the same statements for Cherry maps with the same critical exponent at the boundary of the flat interval, namely $\ell_1=\ell_2$. The proofs of Proposition $1$ and Proposition $2$ in \cite{5aut} rely only on combinatorial arguments and they do not involve the critical exponent of the map. For this reason they can be repeated in our more general case to prove Proposition \ref{Prop:gapsgotozero} and Proposition \ref{Prop:preimbiggap}.

\section{The asymptotics of renormalization}\label{asymrenormalization}
This section explores the asymptotical behavior of the renormalization operator.  
A crucial role will be played by the following sequence whose asymptotics describe the "geometry" of the system near the boundary points of the flat piece. The sequence $\alpha_n$ is defined as
\begin{equation}\label{eq:alphan}
\alpha_n:=\frac{\left|\left[0, x_3\right]\right|}{\left|\left[0, x_4\right]\right|}=\frac{x_{3,n} }{x_{4,n}}=\frac{S_{2,n}S_{3,n}}{1-S_{2,n}+\left(1-S_{1,n}\right)S_{2,n}S_{3,n}}.
\end{equation}
The above formulation of the sequence $\alpha_n$ in the $S_{i,n}$ coordinates can be deduced from Lemma \ref{xtos}.

\subsection{Asymptotics of the distortions}
In this section we show that the diffeomorphic parts of the renormalizations have bounded distortion. We start by giving the definition of distortion.
\begin{defin}
Let $\varphi: N\to N$ be a $\Cuno$ map where $N$ is an interval. If $T\subset N$ is an interval such that $D\varphi(x)\neq 0$ for every $x\in T$, we define the \emph{distortion} of $\varphi$ in $T$ as:
$$\text{dist}(\varphi,T)=\sup_{x,y\in T}\log\frac{|D\varphi(x)|}{|D\varphi(y)|}.$$

\end{defin}
\begin{prop}\label{affdiffeo}
Fix $\delta>0$. Then, for all $f\in\W_{\delta}$ and $n\in\N$,
$$\begin{matrix}
   \text{dist}(\varphi_{n})=O\left(\alpha_{n-2}^{\frac{1}{\ell}}\right) & \text{dist}(\varphi_{n}^{l})=O\left(\alpha_{n-1}^{\frac{1}{\ell}}\right)& 
   \text{dist}(\varphi_{n}^{r})=O\left(\alpha_{n}^{\frac{1}{\ell}}\right)
  \end{matrix}
$$
where $\ell=\max\left\{\ell_1,\ell_2\right\}$ and the constants in the orders depend only on $\delta$.
\end{prop}
The following Proposition is a preparation for proving Proposition \ref{affdiffeo}.
\begin{prop}\label{Koebe}
Fix $\delta>0$. Then for all $f\in\W_{\delta}$ and $\alpha>0$ the following holds. Let $T$ and $M\subset T$ be two intervals and let $S, D$ be the left and the right component of $T\setminus M$ and $n\in\N$. Suppose that
\begin{enumerate}
\item for every $0\leq i\leq n-1$ the intervals $f^i(T)$ are pairwise disjoint,
\item $f^n:T\to f^n(T)$ is a diffeomorphism,
\item $\frac{\left|f^n(M)\right|}{\left|f^n(S)\right|},\frac{\left|f^n(M)\right|}{\left|f^n(D)\right|}<\alpha$.
\end{enumerate}
Then 
$$
\text{dist}\left(f^n(M)\right)=O(\alpha)
$$
where the constant in the order depends only on $\delta$.
\end{prop}
The proof of the previous proposition can be found in \cite{deMelovanStrien}.
We are now ready to prove Proposition \ref{affdiffeo}.
\begin{proof} In this proof we use the notation introduced in Subsection \ref{Fib}.
We start by proving the statement for $\varphi_n^l$. Define 
\begin{itemize}
\item[-] $T=\left[\hat x_{4,n+1},\hat x_{4,n}\right]$,
\item[-] $M=\left[f(x_{3}),\hat x_{3,n}\right]=\left[0,\hat x_{3,n}\right]$,
\item[-] $S=\left[\hat x_{4,n+1}, f(x_{3})\right]=\left[\hat x_{4,n+1}, 0\right]$,
\item[-] $D=\left[\hat x_{3,n},\hat x_{4,n}\right]$.
\end{itemize}
Observe that $$\varphi_{n}^l=Z_{M}f^{q_n-1}.$$ We claim that:
\begin{enumerate}
\item for every $0\leq i\leq q_n-2$ the intervals $f^i(T)$ are pairwise disjoint,
\item $f^{q_n-1}:T\to f^{q_n-1}(T)$ is a diffeomorphism,
\item ${\left|f^{q_n-1}(M)\right|}/{\left|f^{q_n-1}(S)\right|},{\left|f^{q_n-1}(M)\right|}/{\left|f^{q_n-1}(D)\right|}=O\left(\alpha_{n-1}^{\frac{1}{\ell}}\right)$.
\end{enumerate}
Points $1$ and $2$ comes from general properties of circle maps. For point $3$ observe that 
\begin{eqnarray*}
\frac{\left|f^{q_n-1}(M)\right|}{\left|f^{q_n-1}(S)\right|}=\frac{\left|\left[f^{q_n}(x_3),x_3\right]\right|}{\left|\left[f^{-q_{n-1}}(x_4),f^{q_n}(x_3)\right]\right|}.
\end{eqnarray*}
As consequence, under the image of $f$,
\begin{eqnarray*}
\frac{\left|f^{q_n}(M)\right|}{\left|f^{q_n}(S)\right|}=\frac{\left|\left[\hat x_{2,n-1}, 0\right]\right|}{\left|\left[\hat x_{4,n-1},\hat x_{2,n-1}\right]\right|}=O\left(\alpha_{n-1}\right)
\end{eqnarray*}
where we used Proposition \ref{Prop:preimbiggap}.
Hence 
\begin{eqnarray}\label{eq:movers}
\frac{\left|f^{q_n-1}(M)\right|}{\left|f^{q_n-1}(S)\right|}=O\left(\alpha_{n-1}^{\frac{1}{\ell}}\right).
\end{eqnarray}
Observe that, because $U_f= f^{q_n-1}(D)$, by Corollary \ref{Cor:flatlonger} and (\ref{eq:movers}), 

\begin{eqnarray*}
\frac{\left|f^{q_n-1}(M)\right|}{\left|f^{q_n-1}(D)\right|}=O\left( \frac{\left|f^{q_n-1}(M)\right|}{\left|f^{q_n-1}(S)\right|}\right)=O\left(\alpha_{n-1}^{\frac{1}{\ell}}\right).
\end{eqnarray*}
By Proposition \ref{Koebe} we get the desired distortion estimate for $\varphi_n^l$.

To prove the distortion estimate for $\varphi_n$ notice that $\varphi_n=Z_{M}f^{q_{n-1}-1}$ and repeat the previous argument with
\begin{itemize}
\item[-] $T=\left[\hat x_{4,n-1},\hat x_{4,n}\right]$,
\item[-] $M=\left[\hat x_{1,n}, f(x_{3})\right]=\left[\hat x_{1,n}, 0\right]$,
\item[-] $S=\left[\hat x_{4,n-1}, \hat x_{1,n}\right]$,
\item[-] $D=\left[f(x_{3}),\hat x_{4,n}\right]=\left[0, \hat x_{4,n}\right]$.
\end{itemize}

For the distortion estimate of $\varphi_n^r$, take
\begin{itemize}
\item[-] $T=\left[\hat x_{3,n},\hat x_{3,n-2}\right]$,
\item[-] $M=\left[\hat x_{4,n}, \hat x_{2,n-2}\right]$,
\item[-] $S=\left[\hat x_{3,n}, \hat x_{4,n}\right]$,
\item[-] $D=\left[\hat x_{2,n-2},\hat x_{3,n-2}\right]$,
\end{itemize}
and notice that $\varphi_n^r=Z_{M}f^{q_{n}-1}$. We claim that
\begin{enumerate}
\item for every $0\leq i\leq q_n-2$ the intervals $f^i(T)$ are pairwise disjoint,
\item $f^{q_n-1}:T\to f^{q_n-1}(T)$ is a diffeomorphism,
\item ${\left|f^{q_n-1}(M)\right|}/{\left|f^{q_n-1}(S)\right|},{\left|f^{q_n-1}(M)\right|}/{\left|f^{q_n-1}(D)\right|}=O\left(\alpha_{n}^{\frac{1}{\ell}}\right)$.
\end{enumerate}
Points $1$ and $2$ comes from general properties of circle maps. For point $3$ observe that 
\begin{eqnarray*}
\frac{\left| f^{q_n-1}(M)\right|}{\left| f^{q_n-1}(D)\right|}=\frac{\left|\left[x_4, f^{q_{n+1}}(x_3)\right]\right|}{\left|\left[f^{q_{n+1}}(x_3), f^{q_{n-1}}(x_3)\right]\right|}\leq \frac{\left|\left[x_4, f^{q_{n+1}}(x_3)\right]\right|}{\left|\left[f^{-q_{n}}(x_3), f^{-q_{n}}(x_4))\right]\right|}.
\end{eqnarray*}
As consequence, under the image of $f$,
\begin{eqnarray*}
\frac{\left|f^{q_n}(M)\right|}{\left|f^{q_n}(D)\right|}\leq\frac{\left|\left[0,\hat x_{2,n}\right]\right|}{\left|\left[\hat x_{3,n},\hat x_{4,n}\right]\right|}=O\left(\alpha_{n}\right)
\end{eqnarray*}
where we used Proposition \ref{Prop:preimbiggap}.
Hence 
\begin{eqnarray}\label{eq:moverd1}
\frac{\left|f^{q_n-1}(M)\right|}{\left|f^{q_n-1}(D)\right|}=O\left(\alpha_{n}^{\frac{1}{\ell}}\right).
\end{eqnarray}
Observe that, by Corollary \ref{Cor:flatlonger} and (\ref{eq:moverd1}),
\begin{eqnarray*}
\frac{\left|f^{q_n-1}(M)\right|}{\left|f^{q_n-1}(S)\right|}=\frac{\left|\left[x_4, f^{q_{n+1}}(x_3)\right]\right|}{\left|\left[x_3, x_4\right]\right|}\leq \frac{\left|f^{q_n-1}(M)\right|}{\left|f^{q_n-1}(D)\right|}=O\left(\alpha_{n}^{\frac{1}{\ell}}\right).
\end{eqnarray*}
By Proposition \ref{Koebe} we get the desired distortion estimate for $\varphi_n^r$.
\end{proof}

%
%

\subsection{Asymptotics of renormalization}
\begin{lem}\label{prev}
Fix $\delta>0$. Then for all $f\in\W_{\delta}$ and $n\in\N$,
\begin{eqnarray*}
1.&S_{1,n}&=O\left(\alpha_{n+1}^{\frac{1}{\ell}}\right)\\
2.&  S_{2,n}&=O\left(\alpha_{n-1}^{\frac{1}{\ell}}\right)\\
3.&  s_{n}&=O\left(\alpha_{n}^{\frac{\ell_{n+1}-1}{\ell}}\right)\\
4.&S_{1,n}S_{2,n}S_{3,n}&=O\left(\alpha_{n}\right)\\
5.&S_{2,n}S_{3,n}&=\frac{S_{2,n-1}}{\ell_{n}}O\left(1\right)
\end{eqnarray*}
with $\ell=\max\left\{\ell_1,\ell_2\right\}$, $\ell_n=\ell_1$ for $n$ even and $\ell_n=\ell_2$ for $n$ odd. Moreover the constants in the orders depend only on $\delta$.
\end{lem}
\begin{proof}
We prove the lemma assuming that $n$ is even. As a consequence, in a left-sided neighborhood of $[x_{3,n},x_{4,n}]$, the critical exponent is ${\ell_1}$ and in a right-sided neighborhood the critical exponent is ${\ell_2}$. The proof in the case of $n$ odd is exactly the same, one just has to flip the exponents. Observe that, because $\varphi^l_n$ has bounded distortion (see Proposition \ref{affdiffeo}), we have that $S_{1,n}=O\left(\varphi^l_n\left(S_{1,n}\right)\right)$ and by point $2$ of Lemma \ref{changexs},
$$
\left(\varphi^l_n\left(S_{1,n}\right)\right)^{\ell_1}=x_{2,n+1}.
$$ 
Hence
\begin{eqnarray}\label{eq:s1asalpha}
S_{1,n}=O\left(x_{2,n+1}^{\frac{1}{\ell_1}}\right)=O\left(x_{3,n+1}^{\frac{1}{\ell_1}}\right)=O\left(\alpha_{n+1}^{\frac{1}{\ell_1}}\right)=O\left(\alpha_{n+1}^{\frac{1}{\ell}}\right).
\end{eqnarray}
Point $1$ is proved. In order to show point $2$ observe that, by Proposition \ref{Prop:preimbiggap} and Proposition \ref{affdiffeo}, there exist two constants $K_1$ and $K_2$, depending only on $\delta$, such that 
$$
S_{2,n+2}\leq K\frac{\left|\left[x_{4,n+2},x_{3,n}\right]\right|}{\left|\left[x_{3,n+2},x_{3,n}\right]\right|}\leq K_1K_2\frac{\left|\left[f^{-q_{n+1}}(x_4),x_{3}\right]\right|}{\left|\left[f^{-q_{n+1}}(x_3),x_{3}\right]\right|}.
$$
Moreover,
$$
\left(\frac{\left|\left[f^{-q_{n+1}}(x_4),x_{3}\right]\right|}{\left|\left[f^{-q_{n+1}}(x_3),x_{3}\right]\right|}\right)^{\ell_1}=O\left(\frac{x_{3,n+1}}{x_{4,n+1}}\right).
$$
Combining the two previous inequalities, we find

$$
S_{2,n+2}^{\ell_1}=O\left(\frac{x_{3,n+1}}{x_{4,n+1}}\right)=O\left(\alpha_{n+1}\right).
$$
Point $2$ follows. 
By point $5$ of Lemma \ref{ss}, Proposition \ref{affdiffeo} and (\ref{eq:s1asalpha}) we have
$$
s_{n}^{\ell_2-1}=S_{5,n}=S_{1,n-1}^{\ell_2-1}\cdot\left(\frac{\varphi_n^l\left(S_{1,n-1}\right)}{S_{1,n-1}}\right)^{\ell_2-1}=O\left(\alpha_{n}^{\frac{\ell_2-1}{\ell}}\right).
$$
It follows point $3$. 
The proof of point $4$ is a consequence of Proposition \ref{Prop:preimbiggap}. Namely
$$
S_{1,n}S_{2,n}S_{3,n}=\frac{x_{3,n}-x_{2,n}}{1-x_{2,n}}=O\left(\alpha_n\right).
$$
For proving point $5$, observe that, by points $2$ and $3$ of Lemma \ref{ss} we have
\begin{eqnarray*}
S_{2,n}S_{3,n}=\frac{S_{2,n-1}}{\ell_1}\left[\frac{\ell_1}{S_{2,n-1}}\cdot\frac{1-\left(\varphi_{n-1}^{-1}\circ q_{s_{n-1}}^{-1}\right)\left(1-S_{2,n-1}\right)}{1-\left(\varphi_{n-1}^{l}\left( S_{1,n-1}\right)\right)^{\ell_2}}\right].
\end{eqnarray*}
Observe that, by Proposition \ref{Prop:preimbiggap}, the sequence $\left\{S_{2,n}\right\}$  is bounded away from $1$. As a consequence, by Proposition \ref{affdiffeo}, there exists a constant $K>0$, depending only on $\delta$, such that 
\begin{equation}\label{1-varphiq}
1-\left(\varphi_{n-1}^{-1}\circ q_{s_{n-1}}^{-1}\right)\left(1-S_{2,n-1}\right)=\frac{S_{2,n-1}}{Dq_{s_{n-1}}\left(\theta_{n-1}\right)K}
\end{equation}
where $\left|\theta_n-1\right|\leq S_{2,n}$. 
As before, by the fact that the sequences $\left\{S_{1,n}\right\}$ and $\left\{S_{2,n}\right\}$ are bounded away from $1$, see Proposition \ref{Prop:preimbiggap}, by Proposition \ref{affdiffeo} and by point $5$ of Lemma \ref{changexs}, we get
\begin{equation}\label{Dq}
Dq_{s_{n-1}}\left(\theta_{n-1}\right)=\ell_1 K_1
\end{equation}
and
\begin{equation}\label{1-varphil}
\frac{1}{1-\left(\varphi_{n-1}^l\left(S_{1,n-1}\right)\right)^{\ell_2}}=K_2
\end{equation}
where $K_1$ and $K_2$ are positive constants depending only on $\delta$.
Point $5$ follows.
\end{proof}
\begin{prop}\label{ssn}
Fix $\delta>0$. Then for all $f\in\W_{\delta}$ and $n\in\N$, the following holds. If
$R^nf=(S_{1,n}, S_{2,n}, S_{3,n}, S_{4,n}, S_{5,n}, \varphi_n,\varphi^{l}_n,\varphi^{r}_n)$, then
\begin{eqnarray*}
1.& S_{1,n+1}=& \left[1-\left(\frac{\ell_{n+1} S_{1,n}^{\ell_n}}{S_{2,n}}\right)\right]O(1)\\
2.&  S_{2,n+1}=&\frac{ S_{1,n}S_{2,n}S_{3,n}}{\ell_{n+1} S_{5,n}}O(1)\\
3.&  S_{3,n+1}=&\frac{S_{5,n}}{S_{1,n}S_{3,n}}O(1)\\
 4.& S_{4,n+1}=&\frac{S_{1,n}^{\ell_n}}{S_{4,n}}O(1).\\
5.&  S_{5,n+1}=&S_{1,n}^{\ell_{n}-1}O(1).
\end{eqnarray*}

where $\ell_n=\ell_1$ for $n$ even and $\ell_n=\ell_2$ for $n$ odd. Moreover the constants in the orders depend only on $\delta$.
\end{prop}
\begin{proof}
As in the previous lemma, we assume that $n$ is even. As a consequence, in a left-sided neighborhood of $[x_{3,n},x_{4,n}]$, the critical exponent is ${\ell_1}$ and in a right-sided neighborhood the critical exponent is ${\ell_2}$. The proof in the case of $n$ odd is exactly the same, one just has to flip the exponents. Let us start by proving point $1$. By (\ref{1-varphiq}) and (\ref{Dq}), there exists a constant $K>0$, depending only on $\delta$, such that
$$
1-\left(\varphi_n^{-1}\circ q_{s_n}^{-1}\right)\left(1-S_{2,n}\right)=\frac{S_{2,n}}{\ell_2}K.
$$
Finally, by point $1$ of Lemma \ref{ss}, Proposition \ref{affdiffeo} and by the previous estimate we get
\begin{eqnarray}\label{eq:s1s2relation}
 S_{1,n+1}&=& 1-\left(\frac{\ell_2 S_{1,n}^{\ell_1}}{S_{2,n}}\right)\left[\frac{S_{2,n}}{1-\left(\varphi_n^{-1}\circ q_{s_n}^{-1}\right)\left(1- S_{2,n}\right)}\cdot\frac{\left(\varphi_n^{l}\left( S_{1,n}\right)\right)^{\ell_1}}{\ell_2 S_{1,n}^{\ell_1}}\right]
 \\\nonumber &=&1-\left(\frac{\ell_2 S_{1,n}^{\ell_1}}{S_{2,n}}\right)O(1).
\end{eqnarray}
Notice that the previous estimate 
$$
\frac{\left(\varphi_n^{l}\left( S_{1,n}\right)\right)}{ S_{1,n}}=O(1)
$$
proves point $4$ and $5$ by using point $4$ and $5$ of Lemma \ref{ss}.
Let us prove point $2$. By Proposition \ref{affdiffeo} there exists a constant $K_1>0$, depending only on $\delta$, such that
$$
\frac{\ell_2 S_{5,n}\left(\varphi_n^{-1}\circ q_{s_n}^{-1}\right)\left(S_{1,n}S_{2,n}S_{3,n}\right)}{S_{1,n}S_{2,n}S_{3,n}}\leq\frac{\ell_2 S_{5,n}}{Dq_{s_n}\left(0\right)K_1}.
$$
By the fact that the sequence $\left\{S_{1,n}\right\}$ is bounded away from $1$, see Proposition \ref{Prop:preimbiggap}, by Proposition \ref{affdiffeo} and by point $5$ of Lemma \ref{changexs}, we find
$$
Dq_{s_n}\left(0\right)=\ell_2 S_{5,n}K_2
$$
where $K_2$ is a uniform constant.
Finally, by point $2$ of Lemma \ref{ss}, by (\ref{1-varphil}) and by the previous two estimates we find that 
\begin{eqnarray*}
S_{2,n+1}&=&\frac{ S_{1,n}S_{2,n}S_{3,n}}{\ell_2 S_{5,n}}\left[\frac{\ell_2 S_{5,n}\left(\varphi_n^{-1}\circ q_{s_n}^{-1}\right)\left(S_{1,n}S_{2,n}S_{3,n}\right)}{S_{1,n}S_{2,n}S_{3,n}}\cdot\frac{1}{1-\left(\varphi_n^l\left(S_{1,n}\right)\right)^{\ell_1}}\right]\\&=&\frac{ S_{1,n}S_{2,n}S_{3,n}}{\ell_2 S_{5,n}}O(1).
\end{eqnarray*}
Notice that, by (\ref{eq:s1s2relation}), Proposition \ref{affdiffeo} and the fact that the sequences $\left\{S_{1,n}\right\}$ and $\left\{S_{2,n}\right\}$ are bounded away from $1$, there exist two constants $C_1$ and $C_2$, depending only on $\delta$, such that 
\begin{eqnarray}\label{S1S2}
C_1\leq \frac{\ell_2 S_{1,n}^{\ell_1}}{S_{2,n}}\leq C_2.
\end{eqnarray}
By point $5$ of Lemma \ref{prev}, (\ref{S1S2}) and point $5$ of this proposition we get 
\begin{equation}\label{S1S2S3}
\frac{S_{1,n}S_{2,n}S_{3,n}}{s_n^{\ell_2}}=S_{1,n}O(1).
\end{equation}
We are now ready to prove point $3$. Notice that, by Proposition \ref{affdiffeo}, there exists a positive constant $K_3$, depending only on $\delta$, such that 
$$
\left(\varphi_{n}^{-1}\circ q_{s_{n}}^{-1}\right)\left(S_{1,n}S_{2,n}S_{3,n}\right)=\frac{S_{1,n}S_{2,n}S_{3,n}}{Dq_{s_{n}}\left(\zeta_{n}\right)K_3}
$$
where $q_{s_{n}}\left(\zeta_n\right)\leq S_{1,n}S_{2,n}S_{3,n}K_3$. In particular, because $q_{s_{n}}\left(x\right)\geq x^{\ell_2}$, we get 
\begin{equation}\label{zeta}
0<\zeta_n\leq\left(S_{1,n}S_{2,n}S_{3,n}\right)^{\frac{1}{\ell_2}}K_3^{\frac{1}{\ell_2}}
\end{equation}
By Lemma \ref{ss}, (\ref{1-varphiq}) and the previous estimate we get
$$
S_{3,n+1}=\frac{S_{5,n}}{S_{1,n}S_{3,n}}\left[\frac{1}{S_{5,n}}\cdot\frac{\left(\left(1-s_n\right)\zeta_n+s_n\right)^{\ell_2-1}}{\left(\left(1-s_n\right)\theta_n+s_n\right)^{\ell_2-1}}O(1)\right].
$$
Now, by point $2$ and $3$ of Lemma \ref{prev} and by the definition of $S_{5,n}$
\begin{eqnarray*}
S_{3,n+1}&=&\frac{S_{5,n}}{S_{1,n}S_{3,n}}\left[\frac{\left(\left(1-s_n\right)\zeta_n+s_n\right)^{\ell_2-1}}{S_{5,n}}O(1)\right]\\&=&\frac{S_{5,n}}{S_{1,n}S_{3,n}}\left[\left(\left(1-s_n\right)\frac{\zeta_n}{s_n}+1\right)^{\ell_2 -1}O(1)\right].
\end{eqnarray*}
Finally by (\ref{zeta}), (\ref{S1S2S3}) and point $1$ of Lemma \ref{prev}, we find
\begin{eqnarray*}
S_{3,n+1}&=&\frac{S_{5,n}}{S_{1,n}S_{3,n}}O(1).
\end{eqnarray*}
Point $3$ follows.
\end{proof}
%
Let 
\begin{eqnarray*}
L_1&=&\left(\begin{matrix} 1+\frac{1}{\ell_1} & 1& 0& -1\\
 -\frac{1}{\ell_1} & -1& 0& 1\\
 1 & 0& -1& 0\\
 1-\frac{1}{\ell_1} & 0& 0& 0
 \end{matrix}\right),\\
 L_2&=&
 \left(\begin{matrix} 1+\frac{1}{\ell_2} & 1& 0& -1\\
 -\frac{1}{\ell_2} & -1& 0& 1\\
 1 & 0& -1& 0\\
 1-\frac{1}{\ell_2} & 0& 0& 0
 \end{matrix}\right),\\
L_{even} &=&\left(\begin{matrix} \frac{1}{\ell_1}+\frac{1}{\ell_2}+\frac{1}{\ell_1\ell_2} & \frac{1}{\ell_1}& 0& -\frac{1}{\ell_1}\\
 1-\frac{1}{\ell_1}-\frac{1}{\ell_1\ell_2} & 1-\frac{1}{\ell_1}& 0& \frac{1}{\ell_1}-1\\
 \frac{1}{\ell_2} & 1& 1& -1\\
 1-\frac{1}{\ell_1}-\frac{1}{\ell_1\ell_2}+\frac{1}{\ell_2} &  1-\frac{1}{\ell_1}& 0&  \frac{1}{\ell_1}-1
 \end{matrix}\right),\\
 L_{odd}&=&\left(\begin{matrix} \frac{1}{\ell_2}+\frac{1}{\ell_1}+\frac{1}{\ell_1\ell_2} & \frac{1}{\ell_2}& 0& -\frac{1}{\ell_2}\\
 1-\frac{1}{\ell_2}-\frac{1}{\ell_1\ell_2} & 1-\frac{1}{\ell_2}& 0& \frac{1}{\ell_2}-1\\
 \frac{1}{\ell_1} & 1& 1& -1\\
 1-\frac{1}{\ell_2}-\frac{1}{\ell_1\ell_2}+\frac{1}{\ell_1} &  1-\frac{1}{\ell_2}& 0&  \frac{1}{\ell_2}-1
 \end{matrix}\right).
\end{eqnarray*}
For all $n\in\N$ we defined 
\begin{eqnarray}\label{expr:wn}
w_n=\left(\begin{matrix}
\log S_{2,n}\\\log S_{3,n}\\\log S_{4,n}\\\log S_{5,n}
\end{matrix}
\right)=\left(\begin{matrix}
y_{2,n}\\ y_{3,n}\\ y_{4,n}\\ y_{5,n}
\end{matrix}
\right).
\end{eqnarray}

Equation (\ref{S1S2}) allows to eliminate $S_{1,n}$ which asymptotically is determined by $S_{2,n}$. With the notations introduced above, 
the new estimates of Lemma \ref{ssn} obtained by the substitution of $S_{1,n}$ takes the following linear form.
\begin{prop}\label{wn}
Fix $\delta>0$. Then for all $f\in\W_{\delta}$ and $n\in\N$, the sequence $w_n$ has the form
\begin{eqnarray}\label{wncond1}
w_{2n+2}&=&L_{even}w_{2n}+O\left(1\right)\\
\label{wncond2}
w_{2n+3}&=&L_{odd}w_{2n+1}+O\left(1\right)\\
\label{wncond3}
w_{2n+1}&=&L_1w_{2n}+O\left(1\right)\\
\label{wncond4}
w_{2n+2}&=&L_2w_{2n+1}+O\left(1\right).
\end{eqnarray}
 where the constants in the orders depend only on $\delta$.
\end{prop}
\begin{rem}\label{rem:evensequence}
Observe that, because of properties (\ref{wncond3}) and  (\ref{wncond4}) we can restrict our analysis to the even sequence. The asymptotical behavior of the odd sequence can be expressed in the asymptotics of the even sequence and vice versa.
\end{rem}

\begin{lem}\label{eighvector}
The eigenvalues of $L_{even}$ (and $L_{odd}$) are $
\lambda_1=1, \lambda_0=0,$
\begin{eqnarray*}
\lambda_s &=&\frac{1}{2\ell_1\ell_2}\left\{1+\ell_1+\ell_2-\sqrt{\ell_1^2+\ell_2^2+2\ell_1+2\ell_2-
2\ell_1\ell_2+1}\right\}\in (0,1),\\
\lambda_u &=&\frac{1}{2\ell_1\ell_2}\left\{1+\ell_1+\ell_2+\sqrt{\ell_1^2+\ell_2^2+2\ell_1+2\ell_2-
2\ell_1\ell_2+1}\right\}>0.
\end{eqnarray*}
Moreover the eigenvectors of $L_{even}$, $ E_{1}=(e^1_i)_{i=2,3,4,5}$, $ E_{0}=(e^0_i)_{i=2,3,4,5}$, $ E_{s}=(e^s_i)_{i=2,3,4,5}$ and $ E_{u}=(e^u_i)_{i=2,3,4,5}$  corresponding to the eigenvalues $\lambda_1$, $\lambda_0$, $\lambda_s$ and $\lambda_u$ satisfy the following:
\begin{equation}\label{eq:e1andeu}
E_{1}=(0,0,1,0)\text{ and  }\forall i, e^u_i\neq 0.
\end{equation}
The eigenvectors of $L_{odd}$, $L_1E_{1}=(\hat e^1_i)_{i=2,3,4,5}$, $ L_1E_{0}=(\hat e^0_i)_{i=2,3,4,5}$, $ L_1E_{s}=(\hat e^s_i)_{i=2,3,4,5}$ and $L_1 E_{u}=(\hat e^u_i)_{i=2,3,4,5}$ also satisfy  (\ref{eq:e1andeu}).
\end{lem}
Denote by $\Gamma$ the curve defined by the equation $\lambda_u(\ell_1,\ell_2)=1$. The following holds.
\begin{lem}
Let $Q$ be the quadrant defined by $\left(\ell_1,\ell_2\right)\in\left[1,+\infty\right)\times\left[1,+\infty\right)$. Then $Q\setminus\Gamma$ has two connected components,
$$Q\setminus\Gamma=Q_-\cup Q_+$$
satisfying the following
\begin{itemize}
\item[-] $Q_-$ and $Q_+$ are symmetric with respect to the diagonal,
\item[-] $\lambda_u>1$ on $Q_-$ and $\lambda_u<1$ on $Q_+$,
\item[-] for all ray $R_{\theta}$ starting in $(1,1)$, $R_{\theta}\cap Q_-$ and $R_{\theta}\cap Q_+$ are segments.
\end{itemize}
\end{lem}
\begin{proof} We leave the proof of this elementary lemma to the reader. The graph of the curve $\Gamma$ is plotted in Figure \ref{fig:plot}.
\end{proof}

 \begin{figure}
 \centering
  \includegraphics[width=10cm]{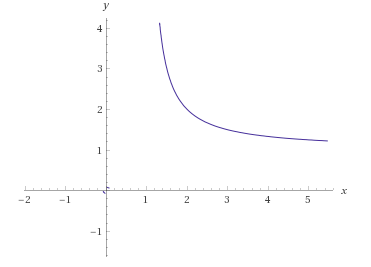}
  \caption{The curve $\Gamma$.}
  \label{fig:plot}
\end{figure}

\section{Statement and Proofs of the Main Theorems}
From Proposition \ref{wn} and Lemma \ref{eighvector} we have that, for all $n\in\N$, there exists $C_u(n)$, $C_s(n)$, $C_1(n)$, $C_0(n)$ such that 
\begin{equation}\label{eq:wnformula}
w_{2n}=C_u(n)E_u+C_s(n)E_s+C_1(n)E_1+C_0(n)E_0.
\end{equation}
Moreover $C_u(n)$, $C_s(n)$, $C_1(n)$ and $C_0(n)$ satisfy the following.
\begin{lem}\label{lem:Cnorders}
Fix $\delta>0$. Then for all $f\in\W_{\delta}$ and $n\in\N$,
$$
C_s(n)=O(1)\text{, }C_0(n)=O(1)\text{ and  }C_1(n)=O(n).
$$
Moreover, if $C_u(n)=O(1)$ then $C_1(n)=O(1)$.
\end{lem}
\begin{proof}
Observe that $C_s(n+1)=\lambda_s C_s(n)+O(1)$ and that $\lambda_s<1$. Similar formulas hold also for $C_0(n+1)$ and for $C_1(n+1)$ using the corresponding eigenvalues. The first statement of the lemma then follows. 

Suppose now that $C_u(n)=O(1)$. Then, by recalling the expression for $w_n$, see (\ref{expr:wn}), by (\ref{eq:wnformula}) and (\ref{eq:e1andeu}), there exists a constant $K$ such that $\frac{1}{K}\leq S_{2,2n},S_{3,2n},S_{5,2n}\leq K$. We use now (\ref{wncond3}) and we find that, the same kind of bounds hold also for $ S_{2,2n-1},S_{3,2n-1}$ and $S_{5,2n-1}$. Moreover, because of (\ref{S1S2}), also the sequence $S_{1,n}$ is bounded from below and from above. Observe now that, by Proposition \ref{Prop:preimbiggap}, $S_{4,2n}$ and $S_{4,2n-1}$ are bounded from above for all $n$. As a consequence, by using point $4$ of Proposition \ref{ssn} and the fact that $S_{1,2n-1}$ is bounded from below, we get a lower bound for $S_{4,2n}$. We have than that $S_{4,2n}$ is bounded from below and from above. Hence, by using again  (\ref{expr:wn}), (\ref{eq:wnformula}) and (\ref{eq:e1andeu}), that $C_1(n)=O(1)$.
\end{proof}
\begin{defin} \label{def:boundeddegenerate}Fix $\delta>0$ and take $f\in\W_{\delta}$.
We say that the sequence $\left\{R^n f\right\}_{n\in\N}$ is bounded if the $\limsup_{n\to\infty }C_u(n)<\infty$.
We say that the sequence $\left\{R^n f\right\}_{n\in\N}$ is degenerate if $\limsup_{n\to\infty }C_u(n)=\infty$.
\end{defin}
\begin{theo}\label{Theo:maintheo}
Fix $\delta>0$. Then for all $f\in\W_{\delta}$ and $n\in\N$, the following holds.
\begin{itemize}
\item[-] If $\ell_1,\ell_2\in Q_+$ then the sequence $\left\{R^n f\right\}_{n\in\N}$ is bounded.
\item[-] If $\ell_1,\ell_2\in Q_-$ then either the sequence $\left\{R^n f\right\}_{n\in\N}$ is bounded or it is degenerate. In particular, there exists $K_{\delta}>0$ such that, if $C_u(0)\geq K_{\delta}$ then the sequence $\left\{R^n f\right\}_{n\in\N}$ is degenerate.
\end{itemize}

\end{theo}
\begin{proof}
Suppose  $\ell_1,\ell_2\in Q_+$. Then, by the equation $C_u(n+1)=\lambda_uC_u(n)+O(1)$ and because $\lambda_u<1$, we have that $C_u(n)$ is bounded by a positive constant. The statement follows. Suppose  $\ell_1,\ell_2\in Q_-$. We start to prove that if $C_u(0)$ is large enough, then the sequence $\left\{R^n f\right\}_{n\in\N}$ is degenerate.
Observe that $C_u(n+1)=\lambda_u C_u(n)+C_{\delta}$ where $C_{\delta}$ is a constant depending only on $\delta$ and $\lambda_u>1$. Define the sequence $\left\{c_n\right\}$ such that, for all $n\in\N$, $C_u(n)=c_n\lambda_u^n$. As a consequence,
$$
c_n\ge c_0-\sum\frac{C_{\delta}}{\lambda^n_u}=C_u(0)-\sum\frac{C_{\delta}}{\lambda^n_u}.
$$
In particular, if $C_u(0)>2\sum\frac{C_{\delta}}{\lambda^n_u}=K_{\delta}$, then for all $n\in\N$,
$$
C_u(n)\ge \lambda_u^n\sum\frac{C_{\delta}}{\lambda^n_u}.
$$
Hence, $\limsup_{n\to\infty}C_u(n)=\infty$ and the sequence $\left\{R^n f\right\}_{n\in\N}$ is degenerate. In order to prove the dichotomy, observe that if $C_u(n)$ is bounded then, the sequence $\left\{R^n f\right\}_{n\in\N}$ is bounded, otherwise there exists $n_0$ such that $C_u(n_0)>K_{\delta}$ and by the previous argument, the sequence $\left\{R^{n} \left(R^{n_0}f\right)\right\}_{n\in\N}$ is degenerate. Hence $\left\{R^{n} f\right\}_{n\in\N}$ is degenerate
\end{proof}
In the following we explore the consequences of the asymptotical behavior of the renormalization operator on the geometry of the system.

\begin{lem}\label{lem:limcun}
Fix $\delta>0$. Then for all $f\in\W_{\delta}$ and $n\in\N$, if $\ell_1,\ell_2\in Q_{-}$, we have that the limit
$$
\lim_{n\to\infty}\frac{C_u(n)}{\lambda_u^n}
$$
exists.
\end{lem}
\begin{proof}
Observe that $C_u(n+1)=\lambda_u C_u(n)+O(1)$. Define the sequence $\left\{c_n\right\}$ such that, for all $n\in\N$, $C_u(n)=c_n\lambda_u^n$. As a consequence
$$
c_{n+1}=c_n+O\left(\frac{1}{\lambda_u^n}\right)
$$
and because $\lambda_u>1$, the sequence $\left\{c_n\right\}$ converges.
\end{proof}
Fix $\delta>0$ and take $f\in W_{\delta}$ with critical exponents $\ell_1,\ell_2\in Q_-$. We denote the limit from Lemma \ref{lem:limcun} by
$$
G_u(f):=\lim\frac{C_u(n)}{\lambda_u^n}\in\R.
$$
We define in the following the non-wandering set\footnote{The non wandering set of a map $f$ is the set of the points $x$ such that for any open neighborhood $V\ni x$ there exists an integer $n>0$ such that the intersection of $V$ and $f^n(V)$ is non-empty.} for a map in our class.
Fix $\delta>0$. For any function $f\in\W_{\delta}$, the non-wandering set of $f$ is 
\begin{equation}\label{Def:nonwandering}
K_f=\S\setminus\cup_{i\geq 0}f^{-i}(U_f)
\end{equation}
 where  $U_f=[x_3,x_4]$ is the flat interval of $f$.
The proof of the next lemma is obtained by following the same arguments as in \cite{5aut} and \cite{MartensdeMeloMendesvanStrien}.
\begin{lem}
The non-wandering set $K_f$ is a Cantor set. 
\end{lem}
The following theorem gives more specific and geometrical consequences of the concepts of "bounded" and "degenerate" behavior of the renormalizations, see Definition \ref{def:boundeddegenerate}.
\begin{theo} \label{Prop:consequencesgeometry}Fix $\delta>0$ and let $f\in\W_{\delta}$. 
If the sequence $\left\{R^n f\right\}_{n\in\N}$ is bounded, then 
\begin{itemize}
\item[-]$|w_n|=O(1)$,
\item[-] $\text{dist}\left(\varphi_n\right), \text{dist}\left(\varphi^l_n\right), \text{dist}\left(\varphi^r_n\right)=O(1)$,
\item[-] the non-wandering set $K_{f}$ has strictly positive Hausdorff dimension.
\end{itemize}
If the sequence $\left\{R^n f\right\}_{n\in\N}$ is degenerate, then $\lambda_u\ge 1$. In particular, if $\lambda_u>1$, then
\begin{itemize}
\item[-]$w_{2n}=G_u(f)\lambda_u^n\left(E_u+o(1)\right)$, 
\item[-] $w_{2n+1}=G_u(f)\lambda_u^n\left(L_1E_u+o(1)\right)$,
\item[-] $G_u(f)<0$,
\item[-] $\alpha_{2n}=e^{G_u(f)\lambda_u^{n}\left(e^2_u+e^3_u+o(1)\right)}$,
\item[-] $\alpha_{2n+1}=e^{G_u(f)\lambda_u^{n}\left(\hat e^2_u+\hat e^3_u+o(1)\right)}$,
\item[-] $\text{dist}\left(\varphi_n\right), \text{dist}\left(\varphi^l_n\right), \text{dist}\left(\varphi^r_n\right)=O\left(e^{G_u(f)\lambda_u^n\left(g+o(1)\right)}\right)$ where $g>0$,
\item[-] the non-wandering set $K_f$ has zero Hausdorff dimension.
\end{itemize}
\end{theo}
\begin{proof}
Suppose that the sequence $\left\{R^n f\right\}_{n\in\N}$ is bounded, then $C_u(n)=O(1)$ and by Lemma \ref{lem:Cnorders}, $C_s(n), C_0(n),C_1(n)=O(1)$. As a consequence, by Remark \ref{rem:evensequence}, $|w_n|=O(1)$. The bounds on the diffeomorphisms are a straight consequence of Proposition \ref{affdiffeo}. The proof of the positivity of the Hausdorff dimension is exactly the same as the one in Theorem 1.5 of \cite{P1} where the author proves the same statement for circle maps with a flat piece and the same order of criticality at the boundary points of it. The positivity of the Hausdorff dimension is only consequence of the fact that the sequence $\alpha_n$ is bounded away form zero, which is now the case also in our more general context.

If the sequence $\left\{R^n f\right\}_{n\in\N}$ is degenerate, by Theorem \ref{Theo:maintheo}, the critical exponents of $f$ are not in $Q_+$ and as a consequence, $\lambda_u\geq 1$. Let us assume that $\lambda_u>1$. By using (\ref{eq:wnformula}) and Lemma \ref{lem:Cnorders}, we get
$$
\lim_{n\to\infty}\frac{w_{2n}}{\lambda_u^n}=G_u(f)E_u.
$$
It follows that 
$$
w_{2n}=G_u(f)\lambda_u^n\left(E_u+o(1)\right) 
$$
and $G_u(f)<0$. The formula for $w_{2n+1}$ comes from Remark \ref{rem:evensequence}. The equality for $\alpha_n$ is consequence of (\ref{eq:alphan}). Proposition \ref{affdiffeo} implies the bounds for the diffeomorphisms. The estimation of the Hausdorff dimension is the same as in Theorem 1.4 of \cite{P1}.
\end{proof}
\begin{rem}
The order terms in Proposition \ref{wn} can be refined to have $O(1)=w^{*}+O\left(\alpha_{2n}^{\frac{1}{\ell}}\right)$ where $\ell=\max\left\{\ell_1,\ell_2\right\}$ and $w^*$ depends only on the critical exponents. As a consequence, if the sequence $\left\{R^n f\right\}_{n\in\N}$ is degenerate and $\lambda_u=1$,
\begin{itemize}
\item[-]$w_{n}=-G n\left(E_u+o(1)\right)$ with $G>0$, 
\item[-] $\alpha_{n}=e^{-g{n}\left(e^2_u+e^3_u+o(1)\right)}$ with $g>0$,
\item[-] the non-wandering set $K_f$ has zero Hausdorff dimension.
\end{itemize}
\end{rem}

\end{document}